\documentclass[11pt]{amsart}

\usepackage{latexsym, amssymb,enumerate}
\usepackage{appendix}

\newcommand{\be}{\begin{equation}}
\newcommand{\ee}{\end{equation}}
\newcommand{\beq}{\begin{eqnarray}}
\newcommand{\eeq}{\end{eqnarray}}

\newcommand{\cM}{\mathcal{M}}

\def\g{\mathfrak g}
\def\H{{\Bbb H}}

\def\R{{\mathfrak R}}

\newtheorem{prop}{Proposition}[section]
\newtheorem{theo}[prop]{Theorem}
\newtheorem{lemm}[prop]{Lemma}

\newtheorem{rema}[prop]{Remark}

\newtheorem{defi}[prop]{Definition}

\def\tr{\mathrm{tr}}
\def\o{\omega}
\def\p{\partial}

\def\a{\alpha}
\def\b{\beta}
\def\g{\gamma}
\def\d{\delta}
\def\k{\kappa}
\def\l{\lambda}

\def\S{\Sigma}
\def\RR{\mathbb R}\def\R{\mathbb R}

\def\H{\mathbb H}

\def\<{\langle}
\def\>{\rangle}
\def\div{{\rm div}}
\def\H{{\mathbb H}}
\def\begeq{\begin{equation}}
\def\endeq{\end{equation}}
\def\p{\partial}

\def\R{\Bbb R}

\def\tr{{\rm tr}}

\def\d{\delta}
\def\a{\alpha}
\def\b{\beta}

\def\odot{\setbox0=\hbox{$\bigcirc$}\relax \mathbin {\hbox
to0pt{\raise.5pt\hbox to\wd0{\hfil $\wedge$\hfil}\hss}\box0 }}

\numberwithin{equation} {section}

\begin{document}

\title[Penrose inequality] {A penrose inequality for graphs over Kottler space}

\author{Yuxin Ge}
\address{Laboratoire d'Analyse et de Math\'ematiques Appliqu\'ees,
CNRS UMR 8050,
D\'epartement de Math\'ematiques,
Universit\'e Paris Est-Cr\'eteil Val de Marne, \\61 avenue du G\'en\'eral de Gaulle,
94010 Cr\'eteil Cedex, France}
\email{ge@u-pec.fr}

\author{Guofang Wang}
\address{ Albert-Ludwigs-Universit\"at Freiburg,
Mathematisches Institut,
Eckerstr. 1,
79104 Freiburg, Germany}
\email{guofang.wang@math.uni-freiburg.de}
\thanks{GW and JW are partly supported by SFB/TR71 ``Geometric partial differential equations'' of DFG}
\author{Jie Wu}
\address{School of Mathematical Sciences, University of Science and Technology
of China Hefei 230026, P. R. China
\and
 Albert-Ludwigs-Universit\"at Freiburg,
Mathematisches Institut
Eckerstr. 1,
79104 Freiburg, Germany
}
\email{jie.wu@math.uni-freiburg.de}
\author{Chao Xia}\address{Max-Planck-Institut f\"ur Mathematik in den Naturwissenschaft, Inselstr. 22, D-04103, Leipzig, Germany}
\begin{abstract} In this work, we prove an optimal  Penrose  inequality for asymptotically locally hyperbolic manifolds which can be realized as graphs 
over Kottler space. Such inequality relies heavily on an optimal weighted Alexandrov-Fenchel inequality for the mean convex star shaped hypersurfaces  in Kottler space.
\end{abstract}

\thanks{Part of this work was done while CX was visiting the mathematical institute of Albert-Ludwigs-Universit\"at Freiburg. He would like to thank the institute for its hospitality.
}
\email{chao.xia@mis.mpg.de}

\maketitle
\section{Introduction}
The famous Penrose inequality (conjecture) in general relativity, as a refinement of the positive mas theorem (\cite{SY}, \cite{W}),
states that the total mass of a space time is no less than the mass of its black holes which are measured by the area of its event horizons. When the cosmological constant $\Lambda=0$,
its Riemannian version  reads that  an asymptotically flat manifold $(\cM^n,g)$  with an outermost minimal boundary $\Sigma$ (a horizon) has the ADM mass
\begin{equation}\label{eq001}
m_{ADM}\geq \frac 12\left(\frac{|\Sigma|}{\omega_{n-1}}\right)^{\frac{n-2}{n-1}},
\end{equation}
provided that the dominant condition $R_g\ge 0$ holds.
Here $R_g$ is the scalar curvature of  $(\cM^n,g)$, $|\Sigma|$ is the area of $\Sigma$ and $\omega_{n-1}$ is the area of the unit $(n-1)$-sphere. Moreover, equality holds if and only if $(\cM,g)$ is isometric to the exterior Schwarzschild solution. For the case $n=3$, (\ref{eq001})  was
  proved by Huisken-Ilmanen \cite{HI} using the inverse mean curvature flow
 and by Bray \cite{BP} using a conformal flow.
Later, Bray's proof was generalized by Bray and Lee \cite{BL} to the case  $n\le 7$.  For related results and further development, see  the excellent surveys \cite{Bray}, \cite{Mars}.
Recently  Lam \cite{Lam}
 gave an elegant proof of $(\ref{eq001})$
for  asymptotically flat graphs over  $\R^{n}$ for all dimensions. 
 His proof was later extended
in \cite{dLG1,HuangWu, MV}. Very recently,
 a general Penrose inequality for a  higher order mass was conjectured in \cite{GWW1}, which is true
 for the graph cases  \cite{GWW1, LWX2} and conformally flat cases  \cite{GWW2}.

In recent years, there has been great interest to extend the previous results to a spacetime with a negative cosmological constant $\Lambda<0$.
In the time symmetric case, $(\cM^n,g)$ is now an asymptotically  hyperbolic  manifold with an outermost minimal boundary $\Sigma$.  In the asymptotically hyperbolic manifolds, a mass-like invariant, which generalizes the ADM mass,  was introduced by
Chru\'sciel, Herzlich and Nagy \cite{CH,CN,Herzlich}. See also an earlier contribution by Wang \cite{Wang} for the special case of conformally compact manifolds. For this mass
$m^\H$ the corresponding Penrose conjecture is
\begin{equation} \label{eq002}
m^{\H}\geq
\frac{1}{2}\left\{
\left( \frac{|\Sigma|}{\omega_{n-1}} \right)^{\frac{n-2}{n-1}}
+ \left(\frac{|\Sigma|}{\omega_{n-1}}\right)^{\frac{n}{n-1}}
\right\},
\end{equation}
provided that the dominant energy condition
$
R_g \geq -n(n-1)
$
holds. This is a  very difficult problem. Neves \cite{Neves} showed that the powerful inverse mean curvature flow of Huisken-Ilmanen
\cite{HI} alone could not work for proving  \eqref{eq002}.
 For the special case that the asymptotically  hyperbolic  manifold can be represented by a graph over the hyperbolic space $\H^n$, Dahl-Gicquaud-Sakovich \cite{DGS}
 and de Lima-Gir\~ao \cite{dLG3} proved this conjecture with a help of a sharp Alexandrov-Fenchel inequality for a weighted mean curvature integral in $\H^n$.
Recently there have been many contributions in establishing Alexandrov-Fenchel inequalities in $\H^n$, see \cite{BHW, GWW3,GWW4,LWX,WX}.
Penrose inequalities for the Gauss-Bonnet-Chern mass have been studied in \cite{GWW1, GWW4}.

In this paper we are interested in studying asymptotically locally hyperbolic (ALH) manifolds.
Let us first introduce the locally hyperbolic  metrics. Fix $\kappa=\pm 1,0$ and suppose $(N^{n-1},\hat g)$ is a closed space form of sectional curvature $\kappa$. Consider
 the product manifold $P_{\kappa}=I_{\kappa}\times N$, where $I_{-1}=(1,+\infty)$ and $I_0=I_1=(0,\infty)$ endowed with the
 warped product metric
\begin{equation}
b_{\kappa}=\frac{d{\rho}^2}{V_{\kappa}^2(\rho)}+{\rho}^2\hat g,\quad \rho\in I_{\kappa},\quad\mbox{and}\quad V_{\kappa}(\rho)=\sqrt{{\rho}^2+\kappa}.
\end{equation}
One can  easily check that the sectional curvature of $(P_{\kappa},b_{\kappa})$ equals to $-1$ and thus it is called {\it locally hyperbolic}. Note that in the case $\kappa=1$ and $(N,\hat g)$ is a round sphere, $(P_{\kappa},b_{\kappa})$ is exactly the hyperbolic space. Since there are a lot of work on the case that $\kappa=1$ and $(N,\hat g)$ is a round sphere, see the work mentioned above, we will in principle focus on the remaining case, the locally hyperbolic case.
Namely, $\kappa=-1, 0$ or $\kappa=1$ and $N$ is a space form other than the standard sphere.   In this case, the mass defined by \eqref{AHmass} below is a geometric invariant.
(See Section 3 in \cite{CH}). In order to define this mass,
we recall from \cite{CH} the following definition of ALH manifolds.

\begin{defi}\label{ALH}
A Riemannian manifold $(\cM^n,g)$ is called asymptotically locally hyperbolic (ALH)  if there exists a compact subset $K$ and a diffeomorphism at infinity $\Phi:\cM\setminus K\rightarrow N\times({\rho}_0,+\infty)$, where ${\rho}_0>1$ such that
\begin{equation}\label{decaytau}
\|(\Phi^{-1})^{\ast}g-b_{\kappa}\|_{b_{\kappa}}+\|\nabla^{b_{\kappa}}\left((\Phi^{-1})^{\ast}g\right)\|_{b_{\kappa}}=O({\rho}^{-\tau } ),\quad\tau>\frac n2,
\end{equation}
and
\begin{equation}\label{condition}
\int_{
\cM}V_{\kappa}\;|R_g+n(n-1)| d V_g<\infty.
\end{equation}
\end{defi}

 Then a mass type invariant of $(\cM^n,g)$ with respect to $\Phi$, which we call ALH mass,  can be defined by
\begin{equation}\label{AHmass}
m_{(\cM,g)}=c_n\lim_{\rho\rightarrow\infty}\int_{N_{\rho}} \bigg(V_{\kappa}(\div^{b_{\kappa}} e-d\,\tr^{b_{\kappa}} e)+(\tr^{b_{\kappa}} e)dV_{\k}-e(\nabla^{b_{\kappa}} V_{\kappa},\cdot)\bigg)\nu d\mu,
\end{equation}
where $e:= (\Phi^{-1})^{\ast}g-b_{\kappa}$, $N_{\rho}=\{\rho\}\times N$, $\nu$ is the outer normal of $N_{\rho}$ induced by $b_{\k}$
and $d\mu$ is the area  element with respect to  the induced metric on $N_{\rho}$, $\vartheta_{n-1}$ is the area of $N$
$$\vartheta_{n-1}=|N| \quad\hbox{ and } \quad
c_n=\frac{1}{2(n-1)\vartheta_{n-1}}.$$

For this mass, there is a corresponding  Penrose conjecture.

\

\noindent{\bf Conjecture 1.} {\it Let $(\cM, g)$ be an ALH manifold with an outermost minimal  horizon $\Sigma$. Then the mass
 \[
m_{(\cM,g)}\ge \frac{1}{2}\left(\left(\frac{|\Sigma|}{\vartheta_{n-1}}\right)^{\frac{n}{n-1}}
+\kappa \left(\frac{|\Sigma|}{\vartheta_{n-1}}\right)^{\frac{n-2}{n-1}}\right),
\]
provided that $\cM$ satisfies the dominant condition
\begin{equation}\label{D}
R_g+n(n-1)\ge 0.\end{equation}
Moreover, equality holds if and only if $(\cM, g)$ is a Kottler space.
}

\

The Kottler space, or Kottler-Schwarzschild space, is an analogue of the Schwarzschild space
in the context of asymptotically locally hyperbolic manifolds  which is introduced as follows. We consider the metric
\begin{equation}\label{eq g}
g_{\kappa,m}=\frac{d{\rho}^2}{V_{\kappa,m}^2(\rho)}+{\rho}^2\hat g, \quad {V_{\kappa,m}}=\sqrt{{\rho}^2+\kappa-\frac{2m}{{\rho}^{n-2}}}.
\end{equation}
Let ${\rho}_{\kappa,m}$ be the largest positive root of $V_{\kappa,m}.$  Then the triple $$\left(P_{\kappa,m}= [{\rho}_{\kappa,m}, +\infty)\times N, \;g_{\kappa,m},\;V_{\kappa,m}\right)$$ is a complete vacuum static data set with the negative cosmological constant $-n$ which satisfies
\begin{equation}\label{static}
\bar\Delta V_{\kappa,m}g_{\kappa,m}-\bar\nabla^2 V_{\kappa,m}+V_{\kappa,m}Ric_{g_{\kappa,m}}=0\quad\mbox{and}\quad R_{g_{\kappa,m}}=-n(n-1).
\end{equation}
We remark here that throughout the all paper, $\bar\Delta $ and $\bar\nabla$ denote the Laplacian and covariant derivative   with respect to the metric $g_{\k,m}$.

Remark that in (\ref{eq g}) if $\kappa\ge 0$, the parameter $m$ is always positive; if $\kappa=-1$, the parameter $m$ can be negative. In fact, $m$ belongs to the following interval
\beq\label{interval}
m\in[m_c,+\infty)\quad\mbox{and}\quad m_{c}=-\frac{(n-2)^{\frac{n-2}{2}}}{n^{\frac n2}}.
\eeq
Comparing with the case of  the asymptotically  hyperbolic, this is a new and interesting situation. The corresponding positive mass theorem looks now like

\

\noindent{\bf Conjecture 2.} {\it Let $(\cM, g)$ be an ALH manifold ($\kappa=-1$ case without boundary). Then the mass
 \[
m_{(\cM,g)} \ge m_c=-\frac{(n-2)^{\frac{n-2}{2}}}{n^{\frac n2}},
\]
provided that $\cM$ satisfies the dominant condition \eqref{D}.

}

\

These problems were first studied by Chru\'sciel-Simon  \cite{CS}. Recently, Lee and Neves \cite{LN} used the powerful inverse mean curvature flow to obtain a Penrose inequality for 3 dimensional conformally compact ALH  manifolds if the mass  $m \le  0$. Roughly speaking, they managed to show that
the inverse mean curvature flow of Huisken and Ilmanen does work for ALH with $\kappa=0, -1$, though Neves \cite{Neves} has previously showed that  it alone does
not work for the asymptotically hyperbolic manifolds, i.e., $\kappa=1$.  Very recently,  de Lima and Gir\~ao \cite{dLG4} proved Conjecture 1 for a class of graphical ALH for all dimenions $n\ge 3$, in the range $m\in [0,\infty).$

Motivated by these work and our previous wok on the Gauss-Bonnet-Chern mass, in this paper we want to to show Conjecture 1 for a class of graphical ALH for all dimensions $n\ge 3$, in the full range
\[m\in [m_c,\infty)=[-\frac{(n-2)^{\frac{n-2}{2}}}{n^{\frac n2}}, \infty).\]

In order to state our results, let us introduce the corresponding Kottler-Schwarzschild  spacetime in general relativity
$$-V_{\kappa,m}^2dt^2+g_{\kappa,m}.$$
We consider its Riemannian version, namely $Q_{\kappa,m}=\R\times P_{\kappa,m}$ with the metric
\begin{equation}
\tilde g_{\kappa,m}=V_{\kappa,m}^2dt^2+g_{\kappa,m}.
\end{equation}
It is well-known  that $\tilde g_{\kappa,m}$ is an Einstein metric, i.e.
$$Ric_{\tilde g_{\kappa,m}}+n\tilde g_{\kappa,m}=0,$$
which actually follows from  \eqref{static}. Now let $m$ be a any fixed number
\[m\in [m_c,\infty).\]
We identity $P_{\kappa,m}$ with the  slice $\{0\}\times P_{\kappa,m} \subset Q_{\kappa,m}$ and consider
 a graphs  over $P_{\kappa,m}$ or over a subset $P_{\kappa,m}\backslash \Omega$, where $\Omega$ is a compact smooth subset containing   $\{0\}\times \partial P_{\kappa,m} $.
An ALH graph associated to
 a smooth function $f: P_{\kappa,m}\backslash \Omega \to \mathbb R$
 is  a manifold $\cM^n$ with the induced metric from  $(Q_{\kappa,m},\tilde g_{\kappa,m})$, i.e.
\begin{equation}\label{graphmetric}
g=V_{\kappa,m}^2(\rho)\bar\nabla f\otimes\bar\nabla f+g_{\kappa,m},
\end{equation}
such that $\cM^n=(P_{\kappa,m}\backslash \Omega,g)$ is ALH  in the sense of Definition \ref{ALH}.

We now state the main results of this paper.

\begin{theo}\label{mainthm1}
Suppose $\cM\subset Q_{\kappa,m}$ is an ALH graph over $P_{\kappa,m}$ with inner boundary $\Sigma$, associated to a function $f:P_{\kappa,m}\backslash \Omega\to \R$.
Assume that $\Sigma$ is in a level set of $f$ and $|\bar\nabla f(x)|\rightarrow\infty$ as
$x\rightarrow\Sigma$.
 Then we have
\begin{equation}\label{eq1_thm1}
m_{(\cM,g)}=m+c_n\int_{\cM}\langle\frac{\partial}{\partial t}, \xi\rangle(R_g+n(n-1))dV_g+c_n\int_{\Sigma}V_{\kappa,m}Hd\mu,
\end{equation}
where $H$ is the mean curvature of $\Sigma$  in $(P_{\kappa,m},g_{\kappa,m})$ and $\xi$ is the unit outer normal of $(\cM,g)$ in $(Q_{\kappa,m},\tilde g_{\kappa,m})$.

Moreover, if in addition the dominant energy condition
$$R_g+n(n-1)\geq 0$$holds,  we have
\beq\label{eq2_thm1}
m_{(\cM,g)}\geq m+c_n\int_{\Sigma}V_{\kappa,m}Hd\mu.
\eeq
\end{theo}
\begin{rema}
For any ALH graph over the whole $P_{\k,m}$, we have
\beq
m_{(\cM,g)}\geq m \ge m_c,
\eeq
provided that the dominant energy condition $R_g+n(n-1)\geq 0$ holds, since in this case
$$ m_{(\cM,g)}=m+c_n\int_{\cM}\langle\frac{\partial}{\partial t}, \xi\rangle(R_g+n(n-1))dV_g\geq m\geq m_c.$$
This can be viewed as a version of the positive mass theorem in this setting. See Conjecture 2.
\end{rema}

Comparing with the work of  \cite{dLG4}, which considers graphs over the local hyperbolic space $P_\kappa$, our setting enables us
to consider the negative mass range.   
In order to obtain a Penrose type inequality, we need to establish
a Minkowski type inequality in the Kottler space. This motivates us to study geometric inequalities in the Kottler space.
The corresponding Minkowski type inequality  is proved in the following Theorem.
\begin{theo}\label{mainthm2}
Let $\Sigma$ be a compact embedded hypersurface which is star-shaped with positive mean curvature in $P_{\kappa, m}$, then we have
\begin{eqnarray}\label{eq1_thm2}
\int_{\Sigma}V_{\kappa,m}Hd\mu&\geq& (n-1)\vartheta_{n-1}\left(\left(\frac{|\Sigma|}{\vartheta_{n-1}}\right)^{\frac{n}{n-1}}-\left(\frac{|\partial P_{\kappa,m}|}{\vartheta_{n-1}}\right)^{\frac{n}{n-1}}\right)\nonumber\\
&&+(n-1)\kappa\vartheta_{n-1}\left(\left(\frac{|\Sigma|}{\vartheta_{n-1}}\right)^{\frac{n-2}{n-1}}-\left(\frac{|\partial P_{\kappa,m}|}{\vartheta_{n-1}}\right)^{\frac{n-2}{n-1}}\right),
\end{eqnarray}
where $\partial P_{\kappa,m}=\{\rho_{\kappa,m}\}\times N.$ Equality holds if and only if $\Sigma$ is a slice.
\end{theo}
Here by star-shaped we mean that $\Sigma$ can be represented as a graph over $N^{n-1}$.

 When $m=0$, i.e, $P_{\kappa, m}=P_{\kappa}$, which is a locally hyperbolic space, Theorem \ref{mainthm2} was proved in \cite{dLG4}. When $m\neq 0$,    $P_{\kappa, m}$
has no  constant curvature. A similar inequality was first proved by Brendel-Hung-Wang in their work on anti-de Sitter Schwarzschild space \cite{BHW}. Our proof of Theorem \ref{mainthm2} uses crucially their work.

One can check easily that for the Kottler space $P_{\kappa,m}$ the area of its horizon $\partial  P_{\kappa,m}$  satisfies
\beq\label{m}
m=\frac 12 \left(\left(\frac{|\partial  P_{\kappa,m} |}{\vartheta_{n-1}}\right)^{\frac{n}{n-1}}+\kappa\left(\frac{|\partial P_{\kappa,m} |}{\vartheta_{n-1}}\right)^{\frac{n-2}{n-1}}\right).
\eeq
Combining (\ref{eq2_thm1}), (\ref{eq1_thm2}) and \eqref{m},  we immediately obtain the Penrose inequality for ALH graphs.
\begin{theo} \label{thm1.5}
If $\cM\subset Q_{\kappa,m}$ is an ALH graph as in Theorem \ref{mainthm1}, so that its horizon $\Sigma\subset (P_{\kappa,m},g_{\kappa,m})$ is star-shaped with positive mean curvature, then \begin{equation} \label{PI}
m_{(\cM,g)}\geq \frac 12 \left(\left(\frac{|\Sigma|}{\vartheta_{n-1}}\right)^{\frac{n}{n-1}}+\kappa\left(\frac{|\Sigma|}{\vartheta_{n-1}}\right)^{\frac{n-2}{n-1}}\right),
\end{equation}
Equality is achieved by  the Kottler space.
\end{theo}

 The Kottler-Schwarzschild space $P_{\kappa, m'}$ can be represented as an ALH graph in $(Q_{\kappa},\tilde g_{\kappa})$ over $P_{\kappa, m}$, if $m' \ge m$.
 One can check that it achieves equality in the Penrose inequality \eqref{PI}.

The rigidity in Theorem \ref{thm1.5} should follow from the argument of Huang-Wu \cite{HuangWu}.
We will return to this problem later.

\section{Kottler-Schwarzschild space}
As stated in the introduction, the Kottler space, or Kottler-Schwarzschild space, is an analogue of the Schwarzschild space
in the setting of asymptotically locally hyperbolic manifolds.  Let $(N^{n-1},\hat g)$ be a closed space form of constant sectional curvature $\k$. Then the $n$-dimensional Kottler-Schwarzschild space $P_{\k,m}= [\rho_{\kappa,m}, \infty)\times N$ is equipped with the metric
\begin{equation}\label{example_g}
g_{\kappa,m}=\frac{d{\rho}^2}{V_{\kappa,m}^2(\rho)}+{\rho}^2\hat g, \quad {V_{\kappa,m}}=\sqrt{{\rho}^2+\kappa-\frac{2m}{{\rho}^{n-2}}}.
\end{equation}
One can realize metric (\ref{example_g}) as a graph over the locally hyperbolic spaces
\begin{equation}
b_{\kappa}=\frac{d{\rho}^2}{V_{\kappa}^2(\rho)}+{\rho}^2\hat g,\quad V_{\kappa}(\rho)=\sqrt{{\rho}^2+\kappa},
\end{equation}
provided $m\ge 0$. Explicitly, one needs to find a function $f=f(\rho)$ satisfying
$$(\rho^2+\k)\left(\frac{\partial f}{\partial\rho}\right)^2=\frac{1}{\rho^2+\k-\frac{2m}{\rho^{ n-2}}}-\frac{1}{\rho^2+\k}.$$
 Let $\rho_0:=\rho_{\k,m}$ be the largest positive root of
$$\phi(\rho):=\rho^2+\k-\frac{2m}{\rho^{n-2}}=0.$$ Then when $\rho$ approaches $\rho_{0}$, we have $\frac{\partial f}{\partial\rho}=O((\rho-\rho_{0})^{-\frac 12}),$
 so that one can solve that
 $$f(\rho)=\int_{\rho_0}^{\rho}\frac{1}{\sqrt{s^2+\k}}\;\sqrt{\frac{1}{s^2+\k-\frac{2m}{s^{n-2}}}-\frac{1}{s^2+\k}}\;\;ds.$$
Its horizon is $\{S_{\rho_{0}}:\rho_{0}^n+\k\rho_{0}^{n-2}=2m\}$ which implies (\ref{m}). Also one can check directly that the AH mass (\ref{AHmass}) of the Kottler space is exactly $m$. 
See Lemma \ref{lemm2.0} below. More generally, the metric (\ref{example_g}) can be realized as a graph over another Kottler-Schwarzschild space  $(P_{\k,m_1}, g_{\k,m_1})$, provided $m\ge m_1$.

Remark that in (\ref{example_g}), in order to have a positive root $\rho_0$, if $\kappa\ge 0$, the parameter $m$ should be always positive; if $\kappa=-1$, the parameter $m$ can be negative. In fact, in this case, $m$ belongs to the following interval
\beq\label{interval}
m\in[m_c,+\infty)\quad\mbox{and}\quad m_{c}=-\frac{(n-2)^{\frac{n-2}{2}}}{n^{\frac n2}}.
\eeq
Here the certain critical value $m_c$ comes from the following. If $m\leq 0$, one can solve the equation
$$\phi'(\rho)=2\rho+(n-2)\frac{2m}{\rho^{n-1}}=0,$$
to get the root $\rho_h=\left(-(n-2)m\right)^{\frac 1n}.$ Note the fact that $\phi(\rho_h)\leq 0$, which yields
$$m\geq-\frac{(n-2)^{\frac{n-2}{2}}}{n^{\frac n2}}.$$

By a change of variable $r=r(\rho)$ with $$r'(\rho)=\frac{1}{V_{\k,m}(\rho)},\quad r(\rho_{\k,m})=0,$$ we can rewrite $P_{\k,m}$ as $P_{\k,m}=[0,\infty)\times N$  equipped with the metric
\begin{eqnarray}
g_{\k,m}:=\bar g:=dr^2+\l_{\k}(r)^2\hat g,
\end{eqnarray}
where $\l_{\k}: [0,\infty)\to [\rho_{\k,m},\infty)$ is the inverse of $r(\rho)$, i.e., $\l_{\k}(r(\rho))=\rho$.

It is easy to check
\begin{eqnarray}\label{l''}
\l'_{\k}(r)&=&V_{\k,m}(\rho)=\sqrt{\k+\l_{\k}(r)^2-2m\l_{\k}(r)^{2-n}},\\
\l''_{\k}(r)&=&\l_{\k}(r)+(n-2)m\l_{\k}(r)^{1-n}.
\end{eqnarray}

By the definition of $\rho_{\k,m},$ we know that $$\l_{\k}''(r)\geq 0 \hbox{ for }r\in[0,\infty).$$

One can also verify
\begin{eqnarray}\label{lr}
\l_{\k}(r)=O(e^{r})\hbox{ as }r\to\infty.
\end{eqnarray}
We take $\k=-1$ as example to verify \eqref{lr}.
\begin{eqnarray*}
r(\rho)&=&\int^\rho_{\rho_{-1,m}} \frac{1}{\sqrt{-1+s^2-2ms^{2-n}}}ds\\%-\int_{s_\k}^\infty\frac{1}{\sqrt{\k+t^2-mt^{2-n}}}dt+\int_1^\infty \frac{1}{\sqrt{\k+t^2}}dt\\
&=& \int^\rho_1  \frac{1}{\sqrt{-1+s^2}}dt+\int^1_{\rho_{-1,m}}  \frac{1}{\sqrt{-1+s^2-2ms^{2-n}}}ds \\&&+\int^\rho_{1}\left(\frac{1}{\sqrt{-1+s^2-2ms^{2-n}}}-\frac{1}{\sqrt{-1+s^2}}\right)dt\\
&=&\ln(2\sqrt{\rho^2-1}+2\rho)-c-\frac{m}{n}\rho^{-n}+O(\rho^{-n-2}) \hbox{ as }\rho\to\infty.
\end{eqnarray*}
Here $c=\ln 2+\int^1_{\rho_{-1,m}}  \frac{1}{\sqrt{-1+s^2-2ms^{2-n}}}ds$.
By Taylor expansion, we have
\begin{eqnarray*}
\frac{e^{r(\rho)+c}}{4}+e^{-(r(\rho)+c)}= (1+o(1))\rho+o(1),
\end{eqnarray*}
which implies $\l_{\k}(r)=\rho=O(e^r)$ as $r\to\infty$.

Let $R_{\a\b\g\d}$ denote the Riemannian curvature tensor in $P_{\k,m}$. Let $\bar{\nabla}$ and $\bar{\Delta}$ denote the covariant derivative and the Laplacian on $P_{\k,m}$ respectively.
The Riemannian and Ricci curvature of $(P_{\k,m}, \bar g)$ are given by
\begin{eqnarray*}
R_{ijkl}&=&\l_\k(r)^2(\k-\l_\k'(r)^2)(\hat g_{ik}\hat g_{jl}-\hat g_{il}\hat g_{jk})=(2m\l_\k^{-n}-1)(\bar g_{ik}\bar g_{jl}-\bar g_{il}\bar g_{jk}),\\
R_{ijkr}&=&0,\\ R_{irjr}&=&-\l_\k(r)\l_\k''(r)\hat g_{ij}=-(1+(n-2)m\l_\k^{-n})\bar g_{ij}.\\
Ric(\bar g)&=&-\left(\frac{\l''_{\k}(r)}{\l_{\k}(r)}-(n-2)\frac{\k-\l'_{\k}(r)^2}{\l_{\k}(r)^2}\right)\bar g\nonumber-(n-2)\left(\frac{\l''_{\k}(r)}{\l_{\k}(r)}+\frac{\k-\l'_{\k}(r)^2}{\l_{\k}(r)^2}\right)dr^2\\
&=&\left(-(n-1)+(n-2)m\l_{\k}(r)^{-n}\right)\bar g -n(n-2)m\l_{\k}(r)^{-n}dr^2.
\end{eqnarray*}
It follows from \eqref{lr} that
\begin{eqnarray}\label{Riem}
|R_{\a\b\g\d}+\bar g_{\a\g}\bar g_{\b\d}-\bar g_{\a\d}\bar g_{\b\g}|_{\bar g}=O(e^{-nr}),\quad |\bar{\nabla}_\mu R_{\a\b\g\d}|_{\hat g}=O(e^{-nr});
\end{eqnarray}
\begin{eqnarray}\label{Ric}
|Ric(\bar g)+(n-1)\bar g|_{\bar g}=O(e^{-nr}).
\end{eqnarray}

\section{The ALH mass of  graphs in the kottler spaces}

First, one can check directly
\begin{lemm}\label{lemm2.0} The Kottler space  $(P_{\kappa, m}, g_{\kappa,m})$ is an  ALH manifold of decay order $n$ with the ALH mass
\[m_{(P_{\k,m},g_{\k,m})}=m.\]
\end{lemm}
Second, instead of computing the ALH mass with $V_\kappa$ in \eqref{condition} one can  compute it with $V_{\kappa, m}$ by using the following Lemma
\begin{lemm}\label{lemm2.1} We have
\begin{equation}
\label{eq_Vkm}
m_{(\cM,g)}
=m\!+\!c_n\lim_{\rho\rightarrow\infty}\int_{N_{\rho}} \bigg(V_{\kappa,m}(\div^{g_{\kappa,m}} \tilde e\!-\!d\tr^{g_{\kappa,m}} \tilde e)\!+\!(\tr^{g_{\kappa,m}} \tilde e)d V_{\kappa,m}\!-\!\tilde e(\nabla^{g_{\kappa,m}} V_{\kappa,m},\cdot)\bigg)\bar\nu d\mu,
\end{equation}
where $\tilde e:=(\Phi^{-1})^{\ast}g-g_{\kappa,m}$ and $\bar\nu$ denotes the outer normal of $N_{\rho}$ induced by $g_{\k,m}$.
\end{lemm}
\begin{proof}
First note that $$e=(\Phi^{-1})^{\ast}g-b_{\kappa}=\tilde e+(g_{\kappa,m}-b_{\kappa}),$$
thus we have
\begin{eqnarray*}
m_{(\cM,g)}
&=&m_{(P_{\k,m},g_{\k,m})}+c_n\lim_{\rho\rightarrow\infty}\int_{N_{\rho}} \bigg(V_{\kappa}({\div}^{b_{\kappa}} \tilde e-d\tr^{b_{\kappa}} \tilde e)+(\tr^{b_{\kappa}} \tilde e)d V_{\kappa}-\tilde e(\nabla^{b_{\kappa}} V_{\kappa},\cdot)\bigg)\nu d\mu\\
&=&m+c_n\lim_{\rho\rightarrow\infty}\int_{N_{\rho}} \bigg(V_{\kappa}({\div}^{b_{\kappa}} \tilde e-d\tr^{b_{\kappa}} \tilde e)+(\tr^{b_{\kappa}} \tilde e)d V_{\kappa}-\tilde e(\nabla^{b_{\kappa}} V_{\kappa},\cdot)\bigg)\nu d\mu.
\end{eqnarray*}
Then using the fact that $g_{\k,m}$ is ALH of decay order $\tau$, one can  replace $V_{\kappa}$ by $V_{\kappa,m}$, $b_{\kappa}$ by $g_{\kappa,m}$ and $\nu$ by $\bar\nu$ in (\ref{AHmass}) without changing mass, that is,
\begin{eqnarray*}
&&\lim_{\rho\rightarrow\infty}\int_{N_{\rho}} \bigg(V_{\kappa}({\div}^{b_{\kappa}} \tilde e-d\tr^{b_{\kappa}} \tilde e)+(\tr^{b_{\kappa}} \tilde e)d V_{\kappa}-\tilde e(\nabla^{b_{\kappa}} V_{\kappa},\cdot)\bigg)\nu d\mu\\
&=&\lim_{\rho\rightarrow\infty}\int_{N_{\rho}} \bigg(V_{\kappa,m}({\div}^{g_{\kappa,m}} \tilde e-d\tr^{g_{\kappa,m}} \tilde e)+(\tr^{g_{\kappa,m}} \tilde e)d V_{\kappa,m}-\tilde e(\nabla^{g_{\kappa,m}} V_{\kappa,m},\cdot)\bigg)\bar\nu d\mu.
\end{eqnarray*}
This implies the desired result.
\end{proof}

According to \cite{M}, the second term in (\ref{eq_Vkm}) is also an integral invariant when the reference metric is taken as the Kottler-Schwarzschild metric $g_{\kappa,m}$ rather than $b_{\kappa}.$
In the spirit of \cite{dLG1,dLG2}, one can estimate the second term since $(P_{\kappa,m}, g_{\kappa,m},V_{\kappa,m})$ satisfies the static equation (\ref{static}). Therefore we can prove Theorem \ref{mainthm1} for the graphs over a Kottler-Schwarzschild space which extends the previous works of graphs over the Euclidean space, hyperbolic space as well as the locally hyperbolic spaces.

\vspace{3mm}

\noindent{\it Proof of Theorem \ref{mainthm1}.}
The proof of this theorem follows in the spirit of the one in \cite{dLG1,dLG2}. For the convenience of readers, we sketch it.
Denote $(\cM,g)\subset (Q_{\kappa,m},\tilde g_{\kappa,m})$ with the unit outer normal $\xi$ and the shape operator $B=-\nabla^{\tilde g_{\kappa,m}}\xi$. Define the Newton tensor inductively by  $$T_r=S_rI-BT_{r-1},\quad T_0=I,$$ where $S_r$ denotes  the $r$-th mean curvature of $(\cM,g)$  with respect to $\xi$.
Let $\{\epsilon_i\}_{i=1}^n$ be a local orthonormal frame on $\cM$, then a direct computation gives (or see (3.3) in \cite{ALM} for the proof)
\beq\label{divTr}
\div_{g}T_r:=\sum_{i=1}^n(\nabla_{e_i}T_r)(e_i)=-B(\div_{g}T_{r-1})-\sum_{i=1}^n(\tilde R(\xi,T_{r-1}(\epsilon_i))\epsilon_i)^T,
\eeq
where $\tilde R$ denotes the curvature tensor of $(Q_{\kappa,m},\tilde g_{\kappa,m})$ and $(\tilde R(\xi,T_{r-1}(\epsilon_i))\epsilon_i)^T$ denotes the tangential component of $\tilde R(\xi,T_{r-1}(\epsilon_i))\epsilon_i$.

Using the fact that $\frac{\partial}{\partial t}$ is a killing vector field, one can check directly (or refer to (8.4) in \cite{ALM} for the proof)
\beq\label{divconformal}
\div_{g}\left(T_r(\frac{\partial}{\partial t})^{T}\right)=\langle\div_{g}T_r,(\frac{\partial }{\partial t})^T\rangle+(r+1)S_{r+1}\langle\frac{\partial}{\partial t},\xi\rangle,
\eeq
where  $(\frac{\partial}{\partial t})^T$ is the tangential component of $\frac{\partial}{\partial t}$ along $\cM$.

Combining (\ref{divTr}) and (\ref{divconformal}) together, we get the following flux-type formula (for $r=1$)
\beq\label{flux}
div_g\left(T_1(\frac{\partial}{\partial t})^{T}\right)=2S_2\langle\frac{\partial}{\partial t}, \xi\rangle+Ric_{\tilde g_{\kappa,m}}(\xi,(\frac{\partial}{\partial t})^T).
\eeq
Denote by $$e_0=(V_{\k,m})^{-1}\frac{\partial}{\partial t}.$$ In the local coordinates $x=(x_1,\cdots,x_n)$ of $(P_{\k,m},g_{\k,m})$, the tangent space $T\cM^n$ is spanned by $$Z_i=(V_{\k,m}\bar\nabla_i f)e_0+\frac{\partial}{\partial x_i},$$
and thus
$$\xi=\frac{1}{\sqrt{1+V_{\k,m}^2|\bar\nabla f|^2}}(e_0-V_{\k,m}\bar\nabla f),$$
which implies
\begin{eqnarray*}
&&(\frac{\partial}{\partial t})^T=V_{\k,m}e_0-\frac{V_{\k,m}}{\sqrt{1+V_{\k,m}^2|\bar\nabla f|^2}}\xi\\
&=&\frac{V_{\k,m}^3|\bar\nabla f|^2}{1+V_{\k,m}^2|\bar\nabla f|^2}e_0+\frac{V_{\k,m}^2}{1+V_{\k,m}^2|\bar\nabla f|^2}\bar\nabla f.
\end{eqnarray*}
On the other hand $(\frac{\partial}{\partial t})^T:=((\frac{\partial}{\partial t})^T)^i Z_i$
which yields
\beq\label{Ti}
((\frac{\partial}{\partial t})^T)^i=\frac{V_{\k,m}^2\bar\nabla^i f}{1+V_{\k,m}^2|\bar\nabla f|^2}.
\eeq
Note that the shape operator  of $\cM^n$ is given by (cf. (4.5) in \cite{GWW4} for instance)
\begin{eqnarray}\label{shape}
 B^i_j&=&\frac{V_{\k,m}}{\sqrt{1+V_{\k,m}^2|\bar\nabla f|^2}}\bigg(\bar\nabla^i\bar\nabla_j f+\frac{\bar\nabla^i f\bar\nabla_j V_{\k,m}}{V_{\k,m}(1+V_{\k,m}^2|\bar\nabla f|^2)}+\frac{\bar\nabla^i V_{\k,m}\bar\nabla_j f}{V_{\k,m}}\\
 &&\qquad\qquad\qquad\qquad\quad-\frac{V^2\bar\nabla^i f\bar\nabla^s f\bar\nabla_s\bar\nabla_j f}{1+V_{\k,m}^2|\bar\nabla f|^2} \bigg).\nonumber
\end{eqnarray}
By the decay property of metric (\ref{graphmetric}) together with (\ref{Ti}), one can check that
\begin{eqnarray}\label{approx}
g_{ij} (T_1(\frac{\partial}{\partial t})^{T})^i\bar\nu^j&\approx& (g_{\k,m})_{ij} (T_1(\frac{\partial}{\partial t})^{T})^i\bar\nu^j\nonumber\\
&=&(T_1)_p^i\frac{V_{\k,m}^2\bar\nabla^p f}{1+V_{\k,m}^2|\bar\nabla f|^2}\bar\nu_i
\approx(T_1)_p^i\frac{V_{\k,m}^2\bar\nabla^p f}{\sqrt{1+V_{\k,m}^2|\bar\nabla f|^2}}\bar\nu_i,
\end{eqnarray}
where $\approx$ means that the two terms differ only by the terms that vanish at infinity after integration.

With  expression (\ref{shape}) and applying the similar argument in the proof of (4.11) in \cite{GWW4}, one can check that
$$V_{\kappa,m}(\bar\nabla^j \tilde e_{ij}-\bar\nabla^i \tilde e_{jj} )-(\tilde e_{ij}\bar\nabla^j V_{\kappa,m}-\tilde e_{jj}\bar\nabla^j V_{\kappa,m})=(T_1)_p^i\frac{V_{\k,m}^2\bar\nabla^p f}{\sqrt{1+V_{\k,m}^2|\bar\nabla f|^2}}.$$
As in the proof of Theorem 1.4 in \cite{GWW4}, integrating by parts gives an extra boundary term that
\begin{eqnarray*}
&&\lim_{\rho\rightarrow\infty}\int_{N_{\rho}} \bigg(V_{\kappa,m}(div^{g_{\kappa,m}} \tilde e\!-\!d\tr^{g_{\kappa,m}} \tilde e)\!+\!(\tr^{g_{\kappa,m}} \tilde e)d V_{\kappa,m}\!-\!\tilde e(\nabla^{g_{\kappa,m}} V_{\kappa,m},\cdot)\bigg)\bar\nu d\mu\\
&=&\lim_{\rho\rightarrow\infty}\int_{N_{\rho}} (T_1)_p^i\frac{V_{\k,m}^2\bar\nabla^p f}{\sqrt{1+V_{\k,m}^2|\bar\nabla f|^2}}\bar\nu_i d\mu+ \int_{\Sigma}V_{\kappa,m}H\bigg(\frac{V_{\k,m}^2|\bar\nabla f|^2}{1+V_{\k,m}^2|\bar\nabla f|^2}\bigg)d\mu.
\end{eqnarray*}
Next using (\ref{approx}) and the assumption that $|\bar\nabla f(x)|\rightarrow\infty$ as
$x\rightarrow\Sigma$, we have
\begin{eqnarray*}
&&\lim_{\rho\rightarrow\infty}\int_{N_{\rho}} (T_1)_p^i\frac{V_{\k,m}^2\bar\nabla^p f}{\sqrt{1+V_{\k,m}^2|\bar\nabla f|^2}}\bar\nu_i d\mu+ \int_{\Sigma}V_{\kappa,m}H\bigg(\frac{V_{\k,m}^2|\bar\nabla f|^2}{1+V_{\k,m}^2|\bar\nabla f|^2}\bigg)d\mu\\
&=&\lim_{\rho\rightarrow\infty}\int_{N_{\rho}} g_{ij} (T_1(\frac{\partial}{\partial t})^{T})^i \bar\nu^j d\mu+ \int_{\Sigma}V_{\kappa,m}Hd\mu.
\end{eqnarray*}
Finally integrating (\ref{flux}) and revoking Lemma \ref{lemm2.1}, we finally obtain
\begin{equation}\label{eq1}
m_{(\cM,g)}=m+c_n\int_{\cM}\left(2S_2\langle\frac{\partial}{\partial t}, \xi\rangle+Ric_{\tilde g_{\kappa,m}}(\xi,(\frac{\partial}{\partial t})^T)\right)dV_g+c_n\int_{\Sigma}V_{\kappa,m}Hd\mu.
\end{equation}
>From the Gauss equation we obtain
$$R_g=R_{\tilde g_{\kappa,m}}-2Ric_{\tilde g_{\kappa,m}}(\xi,\xi)+2S_2.$$
Since $\tilde g_{\kappa,m}$ is an Einstein metric, we have
$$R_g=-n(n-1)+2S_2\quad \mbox{and}\quad Ric_{\tilde g_{\kappa,m}}(\xi,(\frac{\partial}{\partial t})^T)=0.$$
Combining all the things together, we complete the proof of the theorem.
\qed

\section{Inverse mean curvature flow}

Let $\S_0$ be a star-shaped, strictly mean convex closed hypersurface in $P_{\k,m}$ parametrized by $X_0: N\to P_{\k,m}$. Since the case $\k=1$ has been considered in \cite{BHW}, we focus on the case $\k=0$ or $-1$, Consider a family of hypersurfaces $X(\cdot,t): N\to P_{\k,m}$  evolving by the inverse mean curvature flow:
\begin{eqnarray}\label{flow}
\frac{\p X}{\p t}(x,t)=\frac{1}{H(x,t)}\nu(x,t),\quad X(x,0)=X_0(x),
\end{eqnarray}
where $\nu(\cdot,t)$ is the outward normal of $\S_t=X(N,t)$.

Let us first fix the notations. Let $g_{ij}$, $h_{ij}$ and $d\mu$ denote the induced metric, the second fundamental form and the volume element of $\S_t$ respectively. Let $\nabla$ and $\Delta$ denote the covariant derivative and the Laplacian on $\S_t$ respectively. We always use the Einstein summation convention. Let $|A|^2=g^{ij}g^{kl}h_{ik}h_{jl}.$

We collect some evolution equations in the following lemma. For the proof see for instance \cite{Gerhardt}.
\begin{lemm}\label{evolv.lemm}
Along flow (\ref{flow}), we have the following evoltion equations.
\begin{enumerate}[(1)]
\item The volume element of $\S_t$ evolves under
$$\frac{\partial}{\partial t}d\mu=d\mu.$$ Consequently, $$\frac{\partial}{\partial t}|\S_t|=|\S_t|.$$
\item $h_i^j$ evolves under
\begin{eqnarray*}
\frac{\p h_i^j}{\p t} &=&\frac{\Delta h_i^j}{H^2}+\frac{|A|^2}{H^2}h_i^j-\frac{2h_i^kh_k^j}{H}-\frac{2\nabla_i H\nabla^j H}{H^3}\\&&+\frac{1}{H^2}g^{kl}\left(2g^{pj}R_{qikp}h_l^q-g^{pj}R_{qkpl}h_i^q-R_{qkil}h^{qj}+R_{\nu k\nu l}h_i^j\right)\\&&+\frac{1}{H^2}g^{kl}g^{qj}\left(\overline{\nabla}_qR_{\nu kli}+\overline{\nabla}_lR_{\nu ikq}\right)-\frac{2}{H}g^{kj}R_{\nu i \nu k}.
\end{eqnarray*}

\item The mean curvature evolves under
$$\frac{\p H}{\p t} =\frac{\Delta H}{H^2}-2\frac{|\nabla H|^2}{H^3}-\frac{|A|^2}{H}-\frac{Ric(\nu,\nu)}{H}.$$

\item The function $V_{\k,m}$ evolves under
$$\frac{\p}{\p t}V_{\k,m}=\frac{p}{H},$$
where $p:=\langle \overline\nabla V_{\k,m},\nu\rangle$ is the support function of $\Sigma.$

\item The function $\chi=\frac{1}{\<\l_\k\p_r,\nu\>}$ evolves under
\begin{eqnarray}
\frac{\p\chi}{\p t}= \frac{\Delta\chi}{H^2}-\frac{2|\nabla \chi|^2}{\chi H^2}-\frac{|A|^2}{H^2}\chi+\frac{-\chi Ric(\nu,\nu)+\chi^2\l_\k Ric(\nu,\p_r)}{H^2}.
\end{eqnarray}

\item The function $p$, defined above, evolves under
$$\frac{\p p}{\p t}=\frac{\overline{\nabla}^2 V_{\k,m}(\nu,\nu)}{H}+\frac{1}{H^2}\langle\nabla V_{\k,m},\nabla H\rangle,$$
and thus
$$\frac{d}{dt}\int_{\Sigma_t}p d\mu=n\int_{\Sigma_t}\frac{V_{\k,m}}{H}d\mu.$$
\end{enumerate}
\end{lemm}\qed

\

Since $\S_0$ is star-shaped, we can write $\S_0$ as a graph of a function  over $N$: $$\S_0=\{(u_0(x),x): x\in N\}.$$
It is well known that there exists a maximal time interval $[0,T^*)$, $0<T^*\leq \infty$, such that the flow exists and any $X(\cdot,t), t\in [0,T^*)$ are also graphs of functions $u$ over $N$: $$\S_t=\{(u(x,t),x): x\in N\}.$$

Define a function $\varphi(\cdot,t):N\to \RR$ by $$\varphi(x,t)=\int_0^{u(x,t)} \frac{1}{\l_\k(r)}dr.$$ Let $$v=\sqrt{1+|\nabla_{\hat g} \varphi|^2_{\hat g}}.$$
In term of the local coordinates $x^i$ on $N$, the induced metric and the second fundamental form of $\S_t$ are given respectively by
\beq\label{g,h}
g_{ij}=\l_\k^2(\hat g_{ij}+\varphi_i\varphi_j), \quad h_{ij}=\frac{\l_\k}{v}(\l'_{\k}(\hat g_{ij}+\varphi_i\varphi_j)-\varphi_{ij}).
\eeq
Here $\varphi_i=\nabla^{\hat g}_i \varphi$ and $\varphi_{ij}=\nabla^{\hat g}_i\nabla^{\hat g}_j \varphi$.
Thus the mean curvature is given by
\beq\label{meancurv.}
H=g^{ij}h_{ij}=(n-1)\frac{\l'_{\k}}{\l_{\k} v}-\frac{\tilde{g}^{ij}\varphi_{ij}}{\l_{\k} v},
\eeq
where $\tilde{g}^{ij}=\hat g^{ij}-\frac{\varphi^i\varphi^j}{v^2}.$

Along  flow \eqref{flow}, the graph functions $u$ evolve under
\begin{eqnarray}\label{u}
\frac{\p u}{\p t}=\frac{v}{H}.
\end{eqnarray}
Hence
\begin{eqnarray}\label{phi}
\frac{\p \varphi}{\p t}=\frac{v}{\l H}=\frac{v^2}{(n-1)\l'_{\k}-\tilde{g}^{ij}\varphi_{ij}}:=\frac{1}{F(u, \nabla_{\hat g} \varphi, {\nabla_{\hat g}}^2\varphi)}.
\end{eqnarray}

By the parabolic maximum principle, we can derive the $C^0$ and $C^1$ estimates.

\begin{prop}\label{C0}
Let $\underline{u}(t)=\inf_N u(\cdot, t)$ and $\bar{u}(t)=\sup_N u(\cdot, t)$. Then
\begin{eqnarray}
\l_{\k}(\underline{u}(t))\geq e^{\frac{1}{n-1}t}\l_{\k}(\underline{u}(0)),\quad \l_{\k}(\bar{u}(t))\leq e^{\frac{1}{n-1}t}\l_{\k}(\bar{u}(0)).
\end{eqnarray}
\end{prop}

\begin{proof}
At the point where $u(\cdot, t)$ attains its  minimum, we have $v=1$ and $\varphi_{ij}\geq 0$, and hence
 $$H\leq \frac{(n-1)\l_\k'(u)}{\l_\k(u)}.$$
Thus from \eqref{u} we infer that\begin{eqnarray}
\frac{d}{dt}\inf_N \l_\k(\underline{u}(t))\geq (n-1)\l_\k(\underline{u}(t)),
\end{eqnarray}
from which the first assertion follows. The second one is proved in a similar way by considering the maximum point of $u(\cdot, t)$.
\end{proof}

To derive the $C^1$ estimate, we need  to estimate the  upper and lower bounds for $H$.
\begin{prop}\label{estH}We have $H\leq n-1+O(e^{-\frac{1}{n-1}t})$ and $H\geq Ce^{-\frac{1}{n-1}t}$ for some positive constant $C$ depending only on $n, m$ and $\S_0$.
\end{prop}

\begin{proof} By Lemma \ref{evolv.lemm} and \eqref{Ric}, we have
\begin{eqnarray}\label{H}
\frac{\p}{\p t} H^2=\frac{\Delta H^2}{H^2}-\frac{3}{2}\frac{|\nabla H^2|^2}{H^4}-2|A|^2+2(n-1)+O(e^{-nr}).
\end{eqnarray}
In view of the inequality $|A|^2\geq \frac{1}{n-1}H^2$, by using Proposition \ref{C0} and the maximum principle, we deduce\begin{eqnarray*}
\frac{d}{dt}\sup_N H(\cdot,t)^2\leq -\frac{2}{n-1}\sup_N H(\cdot,t)^2+2(n-1)+O(e^{-\frac{n}{n-1}t}).
\end{eqnarray*}
The first assertion follows.

For the second assertion, we take derivative s of \eqref{phi} with respect to $t$ and get
\begin{eqnarray*}
\frac{\p}{\p t}\left(\frac{\p \varphi}{\p t}\right)=-\frac{1}{F^2}\frac{\p F}{\p \varphi_i}\left(\frac{\p \varphi}{\p t}\right)_i-\frac{1}{F^2}\frac{\p F}{\p \varphi_{ij}}\left(\frac{\p \varphi}{\p t}\right)_{ij}-\frac{2(n-1)\l_{\k}\l''_{\k}}{v^2F^2}\frac{\p \varphi}{\p t}.
\end{eqnarray*}
Since $\l''_{\k}(r)\geq 0$, by using the maximum principle, we have
\begin{eqnarray}
\frac{d}{dt}\sup_N\frac{\p \varphi}{\p t}(\cdot,t)\leq 0.
\end{eqnarray}
Taking into account of \eqref{phi} and Proposition \ref{C0},
 we conclude that $$H\geq C\frac{v}{\l_{\k}}\geq Ce^{-\frac{1}{n-1}t}.$$
\end{proof}

\begin{prop}\label{C1} We have $|\nabla_{\hat g} \varphi|_{\hat g}=O(e^{-\frac{1}{(n-1)^2}t})$ and $v=1+O(e^{-\frac{1}{(n-1)^2}t}).$
\end{prop}
\begin{proof}
Let $\o=\frac12 |\nabla_{\hat g} \varphi|_{\hat g}^2$. Since \begin{eqnarray}
\frac{\p \varphi}{\p t}=\frac{v}{\l_{\k} H}:=\frac{1}{F(u, \nabla_{\hat g} \varphi, {\nabla_{\hat g}}^2\varphi)},
\end{eqnarray}
 one can verify that  the evolution equation of $\o$ is
\begin{eqnarray}\label{omega}
&&\frac{\p \o}{\p t}=\frac{\tilde{g}^{ij}}{v^2F^2}\o_{ij}-\frac{1}{F^2}\frac{\p F}{\p \varphi_i}\o_i-\frac{2(n-2)\k}{v^2F^2}\o-\frac{\tilde{g}^{ij}}{v^2F^2}\hat g^{kl}\varphi_{ik}\varphi_{jl}-\frac{2(n-1)\l_{\k}\l_{\k}''}{v^2F^2}\o.
\end{eqnarray}

Notice that $vF=\l H$ and $-\k\leq \l_{\k}^2-2m\l_{\k}^{2-n}$. Using \eqref{l''}, Proposition \ref{C0} and \ref{estH}, we have
\begin{eqnarray}\label{eq}
-\frac{2(n-2)\k}{v^2F^2}-\frac{2(n-1)\l_{\k}\l_{\k}''}{v^2F^2} &\leq &\frac{2(n-2)(\l_{\k}^2-2m\l_{\k}^{2-n})}{\l_{\k}^2H^2}-\frac{2(n-1)(1+(n-2)m\l_{\k}^{-n})}{H^2}\nonumber\\
&= & -\frac{2}{H^2}-\frac{2(n-2)(n+1)m}{\l_{\k}^n H^2}\nonumber\\&\leq & -\frac{2}{(n-1)^2}+Ce^{-\frac{2}{n-1}t}+Ce^{-\frac{n-2}{n-1}t}.
\end{eqnarray}
Thus by using  the maximum principle on \eqref{omega} we have
\begin{eqnarray}
&&\frac{\p }{\p t} \sup_N \o(\cdot, t)\leq \left(-\frac{2}{(n-1)^2}+Ce^{-\frac{2}{n-1}t}\right)\sup_N \o(\cdot, t),
\end{eqnarray}
which implies $\o=O(e^{-\frac{2}{(n-1)^2}t})$. The assertion follows.
\end{proof}

\begin{rema} Proposition \ref{C1} implies that the star-shapedness of $\S_t$ is preserved. Thus as long as the flow exists, we have $\<\p_r,\nu\>>0$ and a graph representation of $\S_t$.
\end{rema}

\begin{prop}\label{est2H} There exists a positive constant $C$ depending only on $n,$ $m$ and $\S_0$, such that $H\geq C$.
\end{prop}

\begin{proof}
Recall the function $\chi=\frac{1}{\<\l(r)\p_r, \nu\>}$. Proposition \ref{C0} and \ref{C1} ensure that $\chi$ is well defined and there exists $C>0$ such that $C^{-1}e^{-\frac{1}{n-1}t}\leq\chi\leq Ce^{-\frac{1}{n-1}t}$.

By Lemma \ref{evolv.lemm} and \eqref{Ric}, we have
\begin{eqnarray}\label{chi}
\frac{\p}{\p t}\log H= \frac{\Delta\log H}{H^2} -\frac{|\nabla \log H|^2}{H^2}-\frac{|A|^2}{H^2}+\frac{n-1}{H^2}+\frac{1}{H^2}O(e^{-nr})
\end{eqnarray} and
\begin{eqnarray}\label{chi}
\frac{\p}{\p t}\log \chi= \frac{\Delta\log\chi}{H^2} -\frac{|\nabla \log\chi|^2}{H^2}-\frac{|A|^2}{H^2}+\frac{1}{H^2}O(e^{-nr}).
\end{eqnarray}
Combining \eqref{H} and \eqref{chi} and using Proposition \ref{C0}, we obtain
\begin{eqnarray*}
\frac{\p}{\p t}(\log\chi-\log H)&=&\frac{\Delta(\log\chi-\log H)}{H^2}\\&&+\frac{\<\nabla (\log H+\log \chi), \nabla (\log H-\log \chi)\>}{H^2}\\&&-\frac{n-1}{H^2}+\frac{Ce^{-\frac{n}{n-1}t}}{H^2}.
\end{eqnarray*}
Using Proposition \ref{estH} and the maximum principle, we have
\begin{eqnarray}
\frac{d}{d t}\sup_N(\log\chi-\log H)(\cdot,t)\leq-\frac{1}{n-1}+Ce^{-\frac{2}{n-1}t}+Ce^{-\frac{n-2}{n-1}t}.
\end{eqnarray}
Hence $e^{\log\chi-\log H}\leq Ce^{-\frac{1}{n-1}t}$. Note that $\chi=\frac{v}{\l}$. Consequently, $H\geq C$.
\end{proof}

With the help of Proposition \ref{est2H}, we are able to improve Proposition \ref{C1}.
\begin{prop}\label{C1'} We have $|\nabla_{\hat g} \varphi|_{\hat g}=O(e^{-\frac{1}{n-1}t})$ and $v=1+O(e^{-\frac{1}{n-1}t}).$
\end{prop}
\begin{proof}We need the following refinement of \eqref{eq}, by taking Proposition \ref{est2H} into account:
\begin{eqnarray*}
-\frac{2(n-2)\k}{v^2F^2}=-\frac{2(n-2)\k}{\l_{\k}^2H^2}\leq Ce^{-\frac{2}{n-1}t};
\end{eqnarray*}
\begin{eqnarray*}
-\frac{2(n-1)\l_{\k}\l_{\k}''}{v^2F^2}&=&-\frac{2(n-1)(1+\frac{n-2}{2}m\l_{\k}^{-n})}{H^2}\\&\leq & -\frac{2}{(n-1)}+Ce^{-\frac{2}{n-1}t}+Ce^{-\frac{n-2}{n-1}t}.
\end{eqnarray*}
Then the proof follows the same way as Proposition \ref{C1}.
 \end{proof}

We  now derive the $C^2$ estimates.
\begin{prop}\label{C2} The second fundamental form $h_{ij}$ is uniformly bounded. Consequently, $|\nabla_{\hat g}^2 \varphi|_{\hat g}\leq C$.
\end{prop}
\begin{proof} Let $M_i^j=Hh_i^j$. By Lemma \ref{evolv.lemm}, we have that $M_i^j$  evolves under
\begin{eqnarray*}
\frac{\p M_i^j}{\p t}&=&\frac{\Delta M_i^j}{H^2}-2\frac{\nabla^k H\nabla_k M_i^j}{H^3}-2\frac{\nabla_i H\nabla^j H}{H^2}\\&&-2\frac{M_i^kM_k^j}{H^2}+\frac{2(n-1)M_i^j}{H^2}+\left(\frac{|M|}{H^2}+1\right)O(e^{-\frac{n}{n-1}t}).
\end{eqnarray*}
Hence the maximal eigenvalue $\mu$ of $M_i^j$ satisfies
\begin{eqnarray}
\frac{\p \mu}{\p t}&=&-2\frac{\mu^2}{H^2}+\frac{2(n-1)\mu}{H^2}+\left(\frac{\mu}{H}+1\right)O(e^{-\frac{n}{n-1}t}).
\end{eqnarray}
 In view of Proposition \ref{estH} and \ref{est2H},  by using  the maximum principle we know that $\mu$ is uniformly bounded from above. Combining the fact $C_1\leq H\leq C_2$, we conclude that $h_i^j$ is uniformly bounded both from above and below.
\end{proof}

Proposition \ref{C0}--\ref{est2H} ensure the uniform parabolicity of equation \eqref{phi}. With the $C^2$ estimates, we can derive  the higher oder estimates via standard parabolic Krylov and Schauder theory, which allows us to obtain the long time existence for the flow.
\begin{prop} The flow \eqref{flow} exists for $t\in [0,\infty)$.
\end{prop}\qed

With Proposition \ref{C0}--\ref{C1'} at hand, we can follow the same argument of Proposition 15 and 16 in \cite{BHW} to obtain improved estimates for $H$ and $h_i^j$.

\begin{prop} \label{prop4.10}$H=n-1+O(te^{-\frac{2}{n-1}t})$ and  $|h_i^j-\delta_i^j|\leq O(t^2e^{-\frac{2}{n-1}t}).$
\end{prop}\qed

Consequently, we have
\begin{prop}\label{C2phi} $|\nabla_{\hat g}^2 \varphi|_{\hat g}\leq O(t^2e^{-\frac{1}{n-1}t}).$
\end{prop}
\begin{proof}
Using Proposition \ref{C0} and \ref{C1'}, we get
\beq\label{eq4_C}
\l_{\k}'=\l_k+O(e^{-\frac{t}{n-1}}),\quad
\frac 1v=O(e^{-\frac{2t}{n-1}}).
\eeq
It follows from Proposition \ref{C2phi} that
\begin{eqnarray}
|h_{ij}-\frac{\l'_\k}{\l_\k v}g_{ij}|_g\leq |h_{ij}-g_{ij}|_g+(n-1)|\frac{\l'_\k}{\l_\k v}-1|\leq O(t^2e^{-\frac{2}{n-1}t}).
\end{eqnarray}
On the other hand, $$g_{ij}=\l_\k^2\hat g_{ij}+\varphi_i\varphi_j=O(e^{\frac{2}{n-1}t})\hat g_{ij}.$$
Thus from  \eqref{g,h} we see
\begin{eqnarray}
|\varphi_{ij}|_{\hat g}= \frac{\l_\k}{v} |h_{ij}-\frac{\l'_\k}{\l_\k v}g_{ij}|_{\hat g}\leq O(t^2e^{-\frac{1}{n-1}t}).
\end{eqnarray}

\end{proof}

If we do more delicate analysis, we may improve the estimates given in 
Proposition \ref{C2phi}
to $o(e^{-\frac 1{n-1} t})$
as in the work of Gerhardt for the inverse mean curvature  flow in $\H^n$.
  Here we avoid to do so, as in the work of Brendle-Hung-Wang \cite{BHW}.
We remark that on a general asymptotically hyperbolic manifolds such estimates may be difficult to
obtain, cf.  the work of Neves \cite{Neves}.  See Remark \ref{rem5.6} below.

\section{Minkowski type inequalities}

We start this section with
\begin{theo}[\cite{BHW}]\label{thm4.1}
Let $\S$ be a compact embedded hypersurface which is star-shaped with positive mean curvature in $(\rho_{\k,m},\infty)\times N^{n-1}$. Let $\Omega$ be the region bounded by $\S$ and the horizon $\p M=\{\rho_{\k,m}\}\times N$. Then
\begin{eqnarray}\label{AF1}
&&\int_\S V_{\k,m} H d\mu\geq  n(n-1) \int_{\Omega}  V_{\k,m} d vol+ (n-1)\k\vartheta_{n-1}\left(\left(\frac{|\S|}{\vartheta_{n-1}}\right)^{\frac{n-2}{n-1}}-\left(\frac{|\p M|}{\vartheta_{n-1}}\right)^{\frac{n-2}{n-1}}\right).
\end{eqnarray}
Equality holds if and only if $\S=\{\rho\}\times N$ for some $\rho\in[\rho_{\k,m},\infty)$.
\end{theo}

When $\k=1$, Theorem \ref{thm4.1} was proved in \cite{BHW}; when $\k=0,-1$, the proof follows from a similar argument,  which is even simpler. For the convenience of the reader,  we include it here.  To prove this theorem, we need the following two lemmas.
\begin{lemm}\label{lemm4.2}
The functional
\begin{eqnarray}
Q_1(t):=\frac{\int_{\S_t} V_{\k,m} H d\mu-n(n-1)\int_{\Omega_t}  V_{\k,m} d vol+(n-1)\k \rho_{\k,m}^{n-2}\vartheta_{n-1}}{|\S_t|^{\frac{n-2}{n-1}}}
\end{eqnarray}
is monotone non-increasing along flow (\ref{flow}).
\end{lemm}
\begin{proof}
The proof of this lemma can be  found in \cite{BHW}.  For completeness, we include the calculations here. To simplify the notation,  we denote $\rho_0=\rho_{\k,m}$.
In view of Lemma \ref{evolv.lemm} and integrating by parts, we calculate
\begin{eqnarray}\label{eq_1_thm4.1}
&&\frac{d}{dt} \int_{\S_t} V_{\k,m} H d\mu\nonumber\\
&=&-\int_{\Sigma_t}\frac 1H\Delta V_{\k,m}d\mu-\int_{\Sigma_t}\frac{V_{\k,m}}{H}(|A|^2+Ric(\nu,\nu))d\mu+\int_{\Sigma_t}(p+V_{\k,m}H)d\mu\nonumber\\
&=&-\int_{\Sigma_t}\frac{V_{\k,m}}{H}|A|^2+\int_{\Sigma_t}(2p+V_{\k,m}H)d\mu\nonumber\\
&\leq &\int_{\S_t}\left(2p+\frac{n-2}{n-1}V_{\k,m} H\right) d\mu,
\end{eqnarray}
where in the third line we used the simple fact $\Delta V_{\k,m}=\overline\Delta V_{\k,m} -\overline\nabla^2V_{\k,m}(\nu,\nu)-Hp$ and (\ref{static}).

Then we use the divergence theorem to deal with the first term that
\begin{eqnarray}\label{eq_2_thm4.1}
\int_{\S_t}p d\mu&=&\int_{\S_t}\langle\overline\nabla V_{\k,m},\nu\rangle d\mu\nonumber\\
&=&\int_{\Omega_t}\bar{\Delta} V_{\k,m}dvol+((n-2)m+\rho_{0}^n)\vartheta_{n-1}\nonumber\\
&=&n\int_{\Omega_t}V_{\k,m}dvol+\left(\frac n2\rho_{0}^n+\frac{n-2}{2}\k\rho_0^{n-2}\right)\vartheta_{n-1},
\end{eqnarray}
where in the last equality we used the relation
$2m=\rho_0^n+\k\rho_0^{n-2}$
 and the fact $\bar\Delta V_{\k,m}=nV_{\k,m}$ which follows from (\ref{static}).

Similarly, by Lemma \ref{evolv.lemm} and (\ref{eq_2_thm4.1}), we have
\beq\label{eq_1'_thm4.1}
\frac{d}{dt}\int_{\Omega_t}  n V_{\k,m}dvol=n\int_{\Sigma_t}\frac{V_{\k,m}}{H}d\mu.
\eeq

 Also a  Heintze-Karcher type
inequality proved by Brendle \cite{Brendle} is needed to estimate the third term, that is,
\beq\label{eq_3_thm4.1}
(n-1)\int_{\S_t}  \frac{V_{\k,m}}{H}d\mu\geq n\int_{\Omega_t}V_{\k,m}dvol+\rho_{0}^n\vartheta_{n-1}.
\eeq
Hence substituting (\ref{eq_2_thm4.1}), (\ref{eq_3_thm4.1}) into (\ref{eq_1_thm4.1}) together with (\ref{eq_1'_thm4.1}), we infer
\begin{eqnarray*}
&&\frac{d}{dt}\left(\int_{\S_t} V_{\k,m} H d\mu-n(n-1)\int_{\Omega_t}  V_{\k,m} d vol\right)\\
&\leq &\int_{\Omega_t} 2n V_{\k,m} dvol+(n\rho_{0}^n+(n-2)\k\rho_0^{n-2})\vartheta_{n-1}\\
&&+\int_{\S_t} \frac{n-2}{n-1}V_{\k,m} H d\mu-\left(n^2\int_{\Omega_t}  V_{\k,m} dvol+n\rho_{0}^n\vartheta_{n-1}\right)\\
&=&\frac{n-2}{n-1}\left(\int_{\S_t} V_{\k,m} H d\mu-n(n-1)\int_{\Omega_t}  V_{\k,m} d vol+(n-1)\k \rho_{0}^{n-2}\vartheta_{n-1}\right).
\end{eqnarray*}
Taking into account of  Lemma \ref{evolv.lemm} (1), we get the assertion.
\end{proof}

\begin{lemm}\label{lemm4.3}
$${\lim \inf}_{t\rightarrow\infty}Q_1(t)\geq (n-1)\k\vartheta_{n-1}^{\frac 1{n-1}}.$$
\end{lemm}
\begin{proof}
%We need to prove that the limit of this quantity is $(n-1)\k\vartheta_{n-1}.$
In view of (\ref{eq_2_thm4.1}), it suffices to prove
\beq\label{aim0}
{\lim\inf}_{t\rightarrow\infty}\frac{\int_{\Sigma_t}V_{\k,m}Hd\mu-(n-1)\int_{\Sigma_t}pd\mu}{|\Sigma_t|^{\frac{n-2}{n-1}}}\geq(n-1)\k\vartheta_{n-1}^{\frac 1{n-1}}.
\eeq
>From (\ref{meancurv.}), Proposition \ref{C1'} and \ref{C2phi}, we have
\beq
H=\frac{1}{v}\left((n-1)\frac{\l'_{\k}}{\l_{\k}}-\frac{1}{\l_{\k}}\Delta_{\hat g}\varphi\right)+O(t^2e^{-\frac{3t}{n-1}}).
\eeq
Using Proposition \ref{C1'} and the expressions  of $\l_{\k},\l_{\k}',$ and $v$, we get
\beq\label{eq4_C}
V_{\k,m}=\l_{\k}'=\l_k\left(1+\frac{\k}{2}(\l_{\k})^{-2}\right)+O(e^{-\frac{4t}{n-1}}),\quad
\frac 1v=1-\frac12|\nabla_{\hat g} \varphi|^2_{\hat g}+O(e^{-\frac{4t}{n-1}})
\eeq
and
\beq\label{eq5_C}
\sqrt{\det g}=\left(\l_{\k}^{n-1}+\frac 12|\nabla_{\hat g}\varphi|_{\hat g}^2\l_{\k}^{n-1}+O(e^{\frac{n-5}{n-1}t})\right)\sqrt{\det\hat g}.
\eeq
Hence we have
\begin{eqnarray}\label{VH}
\int_{\Sigma_t}V_{\k,m}Hd\mu&=&(n-1)\int_{N}(\l_{\k}^n+\k\l_{\k}^{n-2})d\mu_{\hat g}-\int_{N}\l_{\k}^{n-1}\Delta_{\hat g}\varphi d\mu_{\hat g}+O(e^{\frac{n-3}{n-1}t})\nonumber\\
&=&(n-1)\int_{N}(\l_{\k}^n+\k\l_{\k}^{n-2})d\mu_{\hat g}+\int_{N}(n-1)\l_k^{n-4}|\nabla_{\hat g}\l_{\k}|^2d\mu_{\hat g}+O(e^{\frac{n-3}{n-1}t}),
\end{eqnarray}
where in the second line, we have integrated by parts and used the fact
\beq\label{relation0}
|\nabla_{\hat g}\l_{\k}-\l_{\k}^2\nabla_{\hat g}\varphi|_{\hat g}=|\l_k-\l_k'||\nabla_{\hat g} u|_{\hat g}=O(e^{-\frac{t}{n-1}}).
\eeq
Meanwhile, we infer from (\ref{l''}), (\ref{eq4_C}), (\ref{eq5_C}) and (\ref{relation0}) that
\begin{eqnarray}\label{eq6_C}
-\int_{\Sigma_t}pd\mu&=&\int_{\Sigma_t}(V_{\k,m}-\langle\bar\nabla V_{\k,m},\nu\rangle)d\mu-\int_{\Sigma_t}V_{\k,m}d\mu\nonumber\\
&\geq&\int_{\Sigma_t}(V_{\k,m}-|\bar\nabla V_{\k,m}|)d\mu-\int_{\Sigma_t}V_{\k,m}d\mu\nonumber\\
&=&\frac{\k}{2}\int_{N}\l_{\k}^{n-2}d\mu_{\hat g}-\int_{N}\l_{\k}^{n}(1+\frac 12\k\l_{\k}^{-2}+\frac 12\l_{\k}^{-4}|\nabla\l_{\k}|^2)d\mu_{\hat g}+O(e^{\frac{n-3}{n-1}t})\nonumber\\
&=&-\int_{N}\l_{\k}^{n}(1+\frac 12\l_{\k}^{-4}|\nabla\l_{\k}|^2)d\mu_{\hat g}+O(e^{\frac{n-3}{n-1}t})
\end{eqnarray}
 (\ref{VH}) and (\ref{eq6_C}) imply that (\ref{aim0}) is reduced to prove
\beq\label{aim'}
(n-1)\k\int_{N} \l_{\k}^{n-2}+\frac{n-1}{2}\int_{N}\l_{\k}^{n-4}|\nabla \l_{\k}|^2
\geq(n-1)\k {\vartheta_{n-1}}^{\frac 1{n-1}}\left( \int_{N} \l_{\k}^{n-1}\right)^{\frac{n-2}{n-1}}.
\eeq
When $\k=1$, it was already observed in \cite{BHW} that \eqref{aim'} is a non-sharp version of Beckner's Sobolev type inequality, Lemma \ref{Beckner}.  When $\k=-1$, by the H\"{o}lder inequality, we have $$\int_{N} \l_{\k}^{n-2}\leq{\vartheta_{n-1}}^{\frac 1{n-1}}\left( \int_{N} \l_{\k}^{n-1}\right)^{\frac{n-2}{n-1}},$$ which implies
(\ref{aim'}).  When $\k=0$,  \eqref{aim'} is trivial.
Hence we show (\ref{aim0}) and complete the proof.

\end{proof}

\begin{lemm}[\cite{Beckner}]\label{Beckner}
For every positive function $f$ on $\mathbb{S}^{n-1}$, we have
\begin{eqnarray*}
&&(n-1)\int_{\mathbb{S}^{n-1}} f^{n-2}dvol_{\mathbb{S}^{n-1}}+\frac{n-2}{2}\int_{\mathbb{S}^{n-1}} f^{n-4}|\nabla f|^2_{g_{\mathbb{S}^{n-1}}}dvol_{\mathbb{S}^{n-1}}\\
&\geq&(n-1)\omega_{n-1}^{\frac{1}{n-1}}\left(\int_{\mathbb{S}^{n-1}} f^{n-1}dvol_{\mathbb{S}^{n-1}}\right)^{\frac{n-2}{n-1}}.
\end{eqnarray*}
\begin{proof}
Theorem 4 in \cite{Beckner} gives that
\begin{eqnarray*}
&&(n-1)\int_{\mathbb{S}^{n-1}} w^{2}dvol_{\mathbb{S}^{n-1}}+\frac{2}{n-2}\int_{\mathbb{S}^{n-1}} |\nabla w|^2_{g_{\mathbb{S}^{n-1}}}dvol_{\mathbb{S}^{n-1}}\\
&\geq&(n-1)\omega_{n-1}^{\frac{1}{n-1}}\left(\int_{\mathbb{S}^{n-1}} w^{\frac{2(n-1)}{n-2}}dvol_{\mathbb{S}^{n-1}}\right)^{\frac{n-2}{n-1}}.
\end{eqnarray*}
for every positive smooth function $w$. Set $w=f^{\frac{n-2}{2}}$, one gets the desired result.
\end{proof}
\end{lemm}

\begin{rema} \label{rem0}
It is easy to see that the above inequality holds also on the space form $N$.
\end{rema}

\noindent{\it Proof of Theorem \ref{thm4.1}.} Note that $|\p M|=\rho_0^{n-1}\vartheta_{n-1}.$
The inequality \eqref{AF1} follows directly from Lemma \ref{lemm4.2} and Lemma \ref{lemm4.3}. When the equality holds, we have the equality  in  \eqref{eq_1_thm4.1}, which forces $|A|^2=\frac{1}{n-1}H^2$ and hence $\S$ is umbilic. When $m\neq 0$, an umbilic hypersurface must be a slice $\{\rho\}\times N$. When $m=0$, it follows from the equality case in \eqref{aim'} that $\l_\k$ is constant, which implies again $\S$ is a slice $\{\rho\}\times N$.
\qed

\

We now prove another version of Alexandrov-Fenchel inequalities,  which is applicable to prove Penrose inequalities.

\begin{theo}\label{thm5.4}
Let $\S$ be a compact embedded hypersurface which is star-shaped with positive mean curvature in $(\rho_0=\rho_{\k,m},\infty)\times N^{n-1}$. Let $\Omega$ be the region bounded by $\S$ and the horizon $\p M=\{\rho_{0}\}\times N$. Then
\begin{eqnarray*}
\int V_{\k,m}H  d\mu &\ge& (n-1)\kappa \vartheta_{n-1}\left(\left( \frac{|\Sigma|}{\vartheta_{n-1}}\right)^{\frac{n-2}{n-1}}-
\left(\frac{|\partial M|}{\vartheta_{n-1}}\right)^{\frac{n-2}{n-1}} \right) \\
&&+(n-1)\vartheta_{n-1}\left(\left( \frac{|\Sigma|}{\vartheta_{n-1}}\right)^{\frac{n}{n-1}}-
\left(\frac{|\partial M|}{\vartheta_{n-1}}\right)^{\frac{n}{n-1}} \right).
\end{eqnarray*}
Equality holds if and only if $\S=\{\rho\}\times N$ for some $\rho\in[\rho_{\k,m},\infty)$.
%$|\p M|= \rho_{0}^{n-1}\vartheta_{n-1}$, $2m=\k \rho_{0}^{n-2}+\rho_{0}^n$.
\end{theo}

\begin{rema}\label{rem5.6}
 Theorem \ref{thm5.4} should be stronger than Theorem \ref{thm4.1}.  In fact we believe the following
\beq\label{believe}
n\int_{\Omega}  V_{\k,m} d vol\leq\vartheta_{n-1}\left(\left( \frac{|\Sigma|}{\vartheta_{n-1}}\right)^{\frac{n}{n-1}}-
\left(\frac{|\partial M|}{\vartheta_{n-1}}\right)^{\frac{n}{n-1}} \right),
\eeq
holds.
%\end{rema}
%Let $\S$ evolve under flow \eqref{flow}.

\end{rema}

To simplify the notation,  we define
$$\mathbb{J}(\Sigma_t):=n\int_{\Omega_t}  V_{\k,m} d vol\quad\mbox{and}\quad\mathbb{K}(\Sigma_t):=\vartheta_{n-1}\left(\left( \frac{|\Sigma_t|}{\vartheta_{n-1}}\right)^{\frac{n}{n-1}}-
\left(\frac{|\partial M|}{\vartheta_{n-1}}\right)^{\frac{n}{n-1}} \right).$$
 By \eqref{eq_1'_thm4.1} and \eqref{eq_3_thm4.1}, we have
\begin{eqnarray*}
\frac{d}{dt}\int_{\Omega_t}  nV_{\k,m} dvol &=&  \int_{\S_t} n\frac{ V_{\k,m}}{H}d\mu\\
&\geq &\frac{n^2}{n-1}\int_{\Omega_t}  V_{\k,m} dvol+\frac{n}{n-1}\rho_0^n\vartheta_{n-1}.
\end{eqnarray*}
Hence \begin{eqnarray*}
\frac{d}{dt}\left(n\int_{\Omega_t}  V_{\k,m}dvol+\rho_0^n\vartheta_{n-1}\right)\geq \frac{n}{n-1}\left(n\int_{\Omega_t}  V_{\k,m} dvol+\rho_0^n\vartheta_{n-1}\right).
\end{eqnarray*}
Taking into account of  Lemma \ref{evolv.lemm} (1), we find that
\beq\label{monoto.}
\frac{d}{dt}\frac{\mathbb{J}(\Sigma_t)-\mathbb{K}(\Sigma_t)}{\left(\frac{|\S_t|}{\vartheta_{n-1}}\right)^{\frac{n}{n-1}}}\geq0.
\eeq
With this monotonicity, one could prove (\ref{believe}) by analyzing the limit case.  More precisely
if one can prove
\beq\label{limit0}
\lim_{t\to \infty}\frac{\mathbb{J}(\Sigma_t)-\mathbb{K}(\Sigma_t)}{\left(\frac{|\S_t|}{\vartheta_{n-1}}\right)^{\frac{n}{n-1}}}\geq0,
\eeq
then \eqref{believe} follows from \eqref{monoto.}. This is the strategry used in \cite{dLG3}. However, to prove \eqref{limit0} one needs much sharper decay esitmates than that given in Proposition \ref{prop4.10}  and Proposition \ref{C2phi}. Here we would like to avoid to do in this way. Instead we use Theorem \ref{thm4.1}
and the sharp version of the Beckner inequaltiy.
%\end{rema}

\

\noindent{\it Proof of Theorem \ref{thm5.4}.}
It suffices to show when the initial surface $\Sigma$ satisfies \begin{eqnarray}
\mathbb{J}(\Sigma)\leq \mathbb{K}(\Sigma),
\end{eqnarray}
otherwise the assertion follows directly from Theorem \ref{thm4.1}.

By the monotonicity  (\ref{monoto.}),  we divide the proof into two cases.

\noindent\textbf{Case 1:}
there exists some $t_1\in(0,\infty)$ such that
$$\mathbb{J}(\Sigma_{t_1})-\mathbb{K}(\Sigma_{t_1})=n\int_{\Omega_t}  V_{\k,m} dvol+\rho_0^n\vartheta_{n-1}-\vartheta_{n-1}\left(\frac{|\S_{t_1}|}{\vartheta_{n-1}}\right)^{\frac{n}{n-1}}= 0.$$
and $$\mathbb{J}(\Sigma_{t})-\mathbb{K}(\Sigma_{t})=n\int_{\Omega_t}  V_{\k,m} dvol+\rho_0^n\vartheta_{n-1}-\vartheta_{n-1}\left(\frac{|\S_{t}|}{\vartheta_{n-1}}\right)^{\frac{n}{n-1}}\leq 0\hbox{ for }t\in [0,t_1].$$

 From (\ref{eq_2_thm4.1}), we know that
$$\int_{\S_{t}} p d\mu-(n-2)m\vartheta_{n-1}-\o_{n-1}\left(\frac{|\S_{t}|}{\vartheta_{n-1}}\right)^{\frac{n}{n-1}}\leq 0\hbox{ for }t\in [0,t_1].$$

For $t\in [0,t_1]$, by \eqref{eq_1_thm4.1}, we check that

\begin{eqnarray*}
&&\frac{d}{dt}\left(\int_{\S_t} V_{\k,m}Hd\mu+2(n-1)m\vartheta_{n-1}-(n-1)\vartheta_{n-1}\left(\frac{|\S_t|}{\vartheta_{n-1}}\right)^{\frac{n}{n-1}}\right)\\
&\leq&  \frac{n-2}{n-1}\int_{\S_t} V_{\k,m}Hd\mu+2 \int_{\S_t} p d\mu-n\vartheta_{n-1}\left(\frac{|\S_t|}{\vartheta_{n-1}}\right)^{\frac{n}{n-1}}
\\&=&\frac{n-2}{n-1}\left(\int_{\S_t} V_{\k,m}Hd\mu-(n-1)\vartheta_{n-1}\left(\frac{|\S_t|}{\vartheta_{n-1}}\right)^{\frac{n}{n-1}}\right)\\&&+2\int_{\S_t} p d\mu-2\vartheta_{n-1}\left(\frac{|\S_t|}{\vartheta_{n-1}}\right)^{\frac{n}{n-1}}\\&\leq &\frac{n-2}{n-1}\left(\int_{\S_t} V_{\k,m}Hd\mu-(n-1)\vartheta_{n-1}\left(\frac{|\S_t|}{\vartheta_{n-1}}\right)^{\frac{n}{n-1}}+2(n-1)m\vartheta_{n-1}\right).
\end{eqnarray*}

Hence the quantity

$$Q_2(t):=\frac{\int_{\S_t} V_{\k,m}Hd\mu+2(n-1)m\vartheta_{n-1}-(n-1)\vartheta_{n-1}\left(\frac{|\S_t|}{\vartheta_{n-1}}\right)^{\frac{n}{n-1}}}{\left(\frac{|\S_t|}{\vartheta_{n-1}}\right)^{\frac{n-2}{n-1}}}$$ is nonincreasing for $t\in [0,t_1]$.

Using (\ref{m}) and Theorem \ref{thm4.1}, we obtain

$$
Q_2(0)\geq Q_2(t_1)
=Q_1(t_1)\geq (n-1)\k\vartheta_{n-1}.
$$

\noindent\textbf{Case 2:}
For all $t\in[0,\infty)$, we have
$$\mathbb{J}(\Sigma_{t})-\mathbb{K}(\Sigma_{t})\leq 0.$$
>From above, we know that $Q_2(t)$ is monotone non-increasing in $t\in[0,\infty)$.
Thus it suffices to show that
\beq\label{aim}
{\lim \inf}_{t\rightarrow\infty}Q_2(t)\geq (n-1)\k\vartheta_{n-1}^{\frac 1{n-1}}.
\eeq
By the H\"{o}lder inequality and (\ref{eq5_C}) we have
\begin{equation}\label{eq01_C}
 \vartheta_{n-1}\left(\frac{|\Sigma(t)|}{\vartheta_{n-1}}\right)^{n/(n-1)}\le\int_{N}(\sqrt{\det(g)})^{n/(n-1)}=\int_{N} \l_{\k}^{n}(1+\frac{n}{2(n-1)}\l_{\k}^{-4}|\nabla \l_{\k}|^ 2+O(e^{-\frac{4t}{n-1}})).
\end{equation}
Combining (\ref{eq4_C}) and (\ref{eq01_C}), we note that (\ref{aim}) is reduced to prove
\beq\label{aim''}
(n-1)\k\int_{N} \l_{\k}^{n-2}+\frac{n-2}{2}\int_{N}\l_{\k}^{n-4}|\nabla \l_{\k}|^2
\geq(n-1)\k {\vartheta_{n-1}}^{\frac 1{n-1}}\left( \int_{N} \l_{\k}^{n-1}\right)^{\frac{n-2}{n-1}}.
\eeq
When $\k=1$,  \eqref{aim''} follows from  the sharp version of Beckner's Sobolev type inequality on $\mathbb{S}^{n-1}$. See also Remark \ref{rem0}.
When $\k=-1$, by the H\"{o}lder inequality, we have $$\int_{N} \l_{\k}^{n-2}\leq{\vartheta_{n-1}}^{\frac 1{n-1}}\left( \int_{N} \l_{\k}^{n-1}\right)^{\frac{n-2}{n-1}},$$ which implies
(\ref{aim''}).  When $\k=0$,  \eqref{aim''} is trivial.
Hence we show (\ref{aim}). It is easy to show that
equality implies that $\Sigma$ is geodesic.  We complete the proof.

\qed

\end{document}